\documentclass[12pt]{article}
\usepackage{amsmath, pb-diagram}
\usepackage{amssymb}
\usepackage{amscd}
\newenvironment{proof}{\par\noindent{\bf Proof \,}}{$\hfill \Box$\par\bigskip}
\date{\empty}
\newtheorem{thm}{Theorem}[section]
\newtheorem{lem}[thm]{Lemma}

\newtheorem{cor}[thm]{Corollary}

\newtheorem{ex}[thm]{Example}

\begin{document}

\title{Extensions of Square Stable Range One}

\author{H. Chen and M. Sheibani}
\maketitle

\begin{abstract} An ideal $I$ of a ring $R$ is square stable if $aR+bR=R$ with $a\in I,b\in R$ implies that
$a^2+by\in R$ is invertible for a $y\in R$. We prove that an
exchange ideal $I$ of a ring $R$ is square stable if and only if
for any $a\in I$, $a^2\in J(R)$ implies that $a\in J(R)$ if and
only if every regular element in $I$ is strongly regular. Further,
a regular ideal $I$ of a ring $R$ is square stable if and only if
$eRe$ is strongly regular for all idempotents $e\in I$ if and only
if $aR+bR=R$ with $a\in 1+I,b\in R$ implies that $a^2+by\in U(R)$
for a $y\in R$.


\end{abstract}

\section{Introduction}
\vskip4mm Let $R$ be a not necessary unitary ring. Then there is a
canonical unitization $\overline{R}=R\oplus {\Bbb Z}$, with the
multiplication $(a,m)(b,n)=(ab+mb+na, mn)$ for $a,b\in R$ and
$m,n\in {\Bbb Z}$. Evidently, $\overline{R}$ contains $R$ as an
ideal. Recall that a ring with an identity is called to have
square stable range one provided that $aR+bR=R$ implies that
$a^2+by$ is a unit for some $y\in R$. Many interesting properties
of such rings are studied by Dhurana et al. ~\cite{KLW}. The
motivation of this article is to explore the square stable range
one for rings without units. As the preceding observation, we
shall therefore seek a definition which is intrinsic for the
non-unital case.

Let $I$ be an ideal of a ring $R$. We say that $I$ is square stable if
$aR+bR=R$ with $a\in I,b\in R$ implies that $a^2+by\in R$ is
a unit for a $y\in R$. From this, we see that every ideal of a
ring having square stable range one is square stable. Also a ring
$R$ has square stable range one if and only if it is square stable
as an ideal of itself. Let $I$ be an ideal of a commutative ring $R$.
As in square stable range one, we see that $I$ is square
stable if and only if whenever $aR+bR=R$ with $a\in I,b\in R$,
there exists $Y\in M_2(R)$ such that $aI_2+bY\in GL_2(R)$. As many known results on stable range one can not be extended to square stable ideals,
we are focus on those special only for such ideals.

An ideal $I$ of a ring $R$ is an exchange ideal if for any $a\in I$ there exists an idempotent $e\in I$ and $x,y\in I$ such that
$e=ax=a+y-ay$. An ideal $I$ of $R$ is an exchange ideal if and only if
for any $a\in
I$ there exists an idempotent $e\in R$ such that $e\in aR$ and
$1-e\in (1-a)R$~\cite{Ar}. A ring $R$ is an exchange ring if it is exchange as an ideal of itself. Recall that a ring $R$ has
stable range one if $aR+bR=R$ with $a,b\in R$ implies that $a+by\in R$ is
a unit for a $y\in R$. Camillo and Yu proved
that an exchange ring has stable range one if and only if every
regular element in $R$ is unit-regular~\cite[Lemma 1.3.1]{CH}. Here, an element $a\in R$
is (unit) regular if there exists a (unit) $x\in R$ such that
$a=axa$. An element $a\in R$ is called strongly regular if $a\in
a^2R\bigcap Ra^2$. Obviously, $\{$~strongly regular
elements~$\}\subsetneq \{$~unit-regular elements~$\}\subsetneq
\{$~regular elements~$\}$ in a ring $R$. In ~\cite[Theorem 5.8]{KLW}, Khurana
et al. characterized square stable range one for exchange rings,
and they proved that an exchange ring has square stable range one
if and only if every regular element in $R$ is strongly regular. A
natural problem asks that if we can generalize this theorem to square
stable ideals, though the methods of Khurana et al.'s can not be
applied to this case. Fortunately, we see that Khurana-Lam-Wang
Theorem can be extended to such ideals by a completely different
route. We shall prove, in Section 3, that an exchange ideal $I$ of
a ring $R$ is square stable if and only if for any $a\in I$,
$a^2\in J(R)$ implies that $a\in J(R)$ if and only if every
regular element in $I$ is strongly regular.

An ideal $I$ of a ring $R$ is regular if every element
in $I$ is regular. Clearly, every regular ideals of a ring is an
exchange ideals. Recall that $I$ has stable range one if $aR+bR=R$
with $a\in 1+I,b\in R$ implies that $a+by\in R$ is invertible. In
Section 4, we observe that square stable regular ideals possess a
similar characterization. We shall prove that a regular ideal $I$
of a ring $R$ is square stable if and only if $eRe$ is strongly
regular for all idempotents $e\in I$ if and only if $aR+bR=R$ with
$a\in 1+I,b\in R$ implies that $a^2+by\in U(R)$ for a $y\in R$.

Throughout, all rings are associative with an identity, and all
ideals of a ring are two-sided ideals. $J(R)$ and $U(R)$ will
denote the Jacobson radical and the set of all units of a ring
$R$, respectively.

\section{Exchange ideals}

Let $I$ be an ideal of a ring $R$.
If $I\subseteq J(R)$, we easily check that $I$ is a square
stable exchange ideals. Thus, the Jacobson radical and prime radical of any ring are both
square stable exchange ideals~\cite{Ar}. Furthermore, every nil ideal is a square
stable exchange ideal. One easily checks that an ideal $I$ of a ring $R$ is
square stable if and only if for any $a\in I, r\in R$ there exists
$x\in R$ such that $a^2+(1-ar)x\in U(R)$. From this, we claim that
the ring ${\Bbb Z}$ of all integers has no any non-zero square
stable ideal. Let $I=n{\Bbb Z}$ $(n\neq 0,1)$ be a non-zero square
stable ideal of ${\Bbb Z}$. Choose $a=n, r=2n$. Then we have some
$x\in {\Bbb Z}$ such that $a^2+(1-ar)x\in U({\Bbb Z})$. This
implies that $(2x-1)n^2=x\pm 1$. This gives a contradiction as
there is no any integer $n (\neq 0,1)$ satisfying these equations. Here are some pertinent examples.

\begin{ex} Let $R=\left(
\begin{array}{cc}
{\Bbb Z}&{\Bbb Z}\\
0&{\Bbb Z} \end{array} \right)$, and let $I=\left(
\begin{array}{cc}
0&{\Bbb Z}\\
0&0 \end{array} \right)$. Then $I$ is a square stable exchange
ideal of $R$, while $R$ has not square range
one.\end{ex}\begin{proof} Clearly, $I\subseteq J(R)$. Thus, $I$ is
a square stable exchange ideal. As ${\Bbb Z}$ has no square stable range
one~\cite[Proposition 2.1]{KLW}, we easily check that $R$ has no square stable
range one.
\end{proof}

\begin{ex} Let ${\Bbb Z}$ be the ring of all integers. Then $Z[[x]]$ has no square stable one, while $x{\Bbb Z}[[x]]$ is a square stable exchange ideal of ${\Bbb Z}[[x]]$.
\end{ex}\begin{proof} Assume that $Z[[x]]$ has stable range one. Since $3\times 4-11=1$, we can find a $y(x)\in Z[[x]]$ such
that $9-11y(x)\in U\big(Z[[x]]\big)$. We note that $u(x)\in
U\big(Z[[x]]\big)$ if and only if $u(0)=\pm 1$. Hence,
$9-11y(0)=\pm 1$ where $y(0)\in {\Bbb Z}$. But there is no any
integer $y$ as the root of the equations $9-11y=\pm 1$. This gives
a contradiction. Therefore $Z[[x]]$ has no square stable range one.
As $x{\Bbb Z}[[x]]\in J\big({\Bbb Z}[[x]]\big)$, $x{\Bbb Z}[[x]]$ is
a stable range exchange ideal of ${\Bbb Z}[[x]]$, and we are done.
\end{proof}

\begin{lem} \label{31} Let $I$ be a square ideal of a ring $R$, and let $a\in I$. If $a^2\in J(R)$, then $a\in J(R)$.
\end{lem} \begin{proof}
Suppose that $a^2\in J(R)$. For any $x\in R$, we have
$aR+(1-ax)R=R$ with $a\in I$. Since $I$ is square stable, we can
find some $y\in R$ such that $u:=a^2+(1-ax)y\in U(R)$. It follows
by $a^2\in J(R)$ that $1-ax\in R$ is right invertible. This
implies that $a\in J(R)$, as desired.\end{proof}

\begin{lem} \label{32} Let $R$ be a ring and $ax+b=1$ with $a,x,b\in R$. If $a\in R$ is unit-regular, then $a+by\in U(R)$ for a $y\in R$.
\end{lem} \begin{proof}
This is obvious as in the proof of~\cite[Lemma 1.3.1]{CH}.\end{proof}

\begin{thm} \label{33} Let $I$ be an exchange ideal of a ring $R$. Then the following are equivalent:
\end{thm}
\begin{enumerate}
\item [(1)]{\it $I$ is square stable.}
\vspace{-.5mm}
\item [(2)]{\it For any $a\in I$, $a^2\in J(R)\Longrightarrow a\in J(R)$.}
\end{enumerate}
\begin{proof}  $(1)\Rightarrow (2)$ This is obvious by Lemma~\ref{31}.

$(2)\Rightarrow (1)$ Let $f\in R$ be an idempotent. Then
$\big(fa(1-f)\big)^2=0\in J(R)$. By hypothesis, $fa(1-f)\in J(R)$.
Hence, $\overline{fa}=\overline{faf}$. Likewise,
$\overline{af}=\overline{faf}$. Thus,
$\overline{fa}=\overline{af}$ in $R/J(R)$.

Now let $ax+b=1$ for $a\in I, x, b\in R$ then $b=1-ax\in 1+I$. As
$I$ is an exchange ideal, there exists an idempotent $e \in R$
such that $e=bs$ and $1-e=(1-b)t$ for some $s, t\in R$. This
implies that $axt+e=1$, and so $(1-e)axt(1-e)a =(1-e)a$. Thus,
$\overline{(1-e)a(1-e)xt+e}=\overline{1}$ in $R/J(R)$, and then
$\overline{\big((1-e)a\big)^2(xt)^2}+\overline{e}=\overline{1}$.
This shows that
$\overline{((1-e)a)^2}=\overline{\big((1-e)a\big)^2(xt)^2((1-e)a)^2}$.
Set $c=(xt)^2((1-e)a)^2(xt)^2$. Then
$$\overline{((1-e)a)^2}=\overline{\big((1-e)a\big)^2c((1-e)a)^2}~\mbox{and}~\overline{c}=\overline{c\big((1-e)a\big)^2c}.$$
Accordingly,
$\overline{((1-e)a)^2}=\overline{\big((1-e)a\big)^2cd},$
where $d=1-\big((1-e)a\big)^2c+\big((1-e)a\big)^2$. One easily
checks that
$\big(\overline{d}\big)^{-1}=\overline{1-\big((1-e)a\big)^2c+c}\in
R/J(R)$. Hence, $\overline{((1-e)a)^2}\in R/J(R)$ is unit-regular.
By virtue of Lemma~\ref{32}, there exists a $y\in R$ such that
$\overline{(1-e)a^2}+\overline{ey}\in U\big(R/J(R)\big)$. As every
unit lifts modulo $J(R)$, we have a $u\in U(R)$ such that
$(1-e)a^2+ey=u+r$ for a $r\in J(R)$. Therefore
$a^2+bs(y-a^2)=a^2+e(y-a^2)\in U(R)$, as desired.
\end{proof}

As an immediate consequence of Theorem~\ref{33}, we see
that an exchange ring $R$ has square stable range one if and only if $R/J(R)$ is reduced if
and only if $R/J(R)$ is abelian~\cite[Theorem 4.4]{KLW}.

\begin{cor} Let $I$ be an exchange ideal of a ring $R$. Then the following are equivalent:
\end{cor}
\begin{enumerate}
\item [(1)]{\it $I$ is square stable.}
\vspace{-.5mm}
\item [(2)]{\it For any $a\in I$ and idempotent $e\in R$, $ae-ea\in J(R)$.}
\end{enumerate}
\begin{proof} $(1)\Rightarrow (2)$ For any $a\in I$ and idempotent $e\in R$, we see that
$\big(ea(1-e)\big)^2=0\in J(R)$. It follows by Theorem~\ref{33}
that $ea-eae\in J(R)$. Likewise, $eae-ae\in J(R)$. Hence,
$ae-ea=(ae-eae)+(eae-ea)\in J(R)$.

$(2)\Rightarrow (1)$ Given $ax+b=1$ with $a\in I,x,b\in R$, there
exists an idempotent $e\in R$ such that $e=bs$ and $1-e=(1-b)t$
for some $s,t\in R$. Hence, $(1-e)axt+e=1$. By hypothesis, we have
some $r\in J(R)$ such that $(1-e)a=(1-e)a(1-e)+r$. Hence,
$(1-e)a^2(xt)^2+e=1-rxt\in U(R)$. Hence,
$(1-e)a^2(xt)^2(1-rxt)^{-1}+e(1-rxt)^{-1}=1$. This shows that
$e(1-rxt)^{-1}=e$, and so $(1-e)a^2(xt)^2(1-rxt)^{-1}+e=1$. It
follows that $(1-e)a^2(xt)^2(1-rxt)^{-1}(1-e)a^2=(1-e)a^2$. As
$(1-e)a^2(xt)^2(1-rxt)^{-1}\in R$ is an idempotent, we see that
$\overline{(1-e)a^2}\in
\big(\overline{(1-e)a^2}\big)^2\big(R/J(R)\big)\bigcap
\big(R/J(R)\big)\big(\overline{(1-e)a^2}\big)^2$, and so
$\overline{(1-e)a^2}\in R/J(R)$ is strongly regular. Hence, it is
unit-regular~\cite{MM}. As $I$ is an exchange
ideal, we have an idempotent $f\in R$ and a unit $u\in R$ such
that $\overline{(1-e)a^2}=\overline{fu}$. Thus, we can find some
$s\in J(R)$ such that
$fu(xt)^2(1-rxt)^{-1}+e=1-s(xt)^2(1-rxt)^{-1}.$ It follows that
$fu(xt)^2(1-rxt)^{-1}\big(1-s(xt)^2(1-rxt)^{-1}\big)^{-1}+e\big(1-s(xt)^2(1-rxt)^{-1}\big)^{-1}=1.$
As $fu\in R$ is unit-regular, it follows by Lemma~\ref{32} that
$fu+e\big(1-s(xt)^2(1-rxt)^{-1}\big)^{-1}z\in U(R).$ Therefore
$(1-e)a^2+e\big(1-s(xt)^2(1-rxt)^{-1}\big)^{-1}z\in U(R).$ Consequently,
$a^2+bs\big((1-s(xt)^2(1-rxt)^{-1}\big)^{-1}z)-a^2\big)\in
U(R),$ hence the result.\end{proof}

The following result will play an important role in the proof of the main result in this section.

\begin{thm} \label{35} Let $I$ be an exchange ideal of a ring $R$. Then the following are equivalent:
\end{thm}
\begin{enumerate}
\item [(1)]{\it $I$ is square stable.} \vspace{-.5mm}
\item [(2)]{\it For any regular $a\in I$, $\overline{a}\in R/J(R)$ is strongly regular.}
\end{enumerate}
\begin{proof}  $(1)\Rightarrow (2)$ Let $a\in R$ be regular. Write
$a=axa$ for some $x\in R$. Since $ax+(1-ax)=1$, we can find a
$y\in R$ such that $a^2+(1-ax)y=u\in U(R)$. Hence,
$a^2=axa^2=ax\big(a^2+(1-ax)y\big)=axu$. Thus, $ax=a^2u$, and so
$a=(ax)a=a^2ua$. Therefore $a\in a^2R$, and so $\overline{a}\in
\big(\overline{a}\big)^2\big(R/J(R)\big)$. Write $a=a^2x$ for some
$x\in R$. Hence, $a^2(a-xa^2)=0$. This shows that
$(a-xa^2)^3=a(a-xa^2)(a-xa^2)=a^2(a-xa^2)=0$. Hence,
$(a-xa^2)^4=0\in J(R)$. By using Theorem~\ref{33}, $(a-xa^2)^2\in
J(R)$, and so $a-xa^2\in J(R)$. This shows that $\overline{a}\in
\overline{a^2}\big(R/J(R)\big)\bigcap
\big(R/J(R)\big)\overline{a^2}$. That is, $\overline{a}\in R/J(R)$
is strongly regular.

$(2)\Rightarrow (1)$ Given $bc+d=1$ with $b\in I,c,d\in R$, we see
that $d\in 1+I$, and so we can find an idempotent $e\in R$ such
that $e=ds$ and $1-e=(1-d)t$ for some $s,t\in R$. Hence,
$bct+e=1$. Let $p=b(ct)b$. Then
$p(ct)p=b(ct)b(ct)b(ct)b=b(ct)b=p.$ That is, $p\in I$ is
regular. By hypothesis, we have some $s\in R$ such that
$\overline{p}=\overline{p^2s}$ in $R/J(R)$. Hence,
$\overline{p^2(sct)+e}=\overline{p(ct)+e}=\overline{b(ct)b(ct)+e}=\overline{(1-e)+e}=\overline{1}.$ Since
$p^2=(bctb)p=(1-e)(bp)=bp-ebp\in b^2R+eR$, we see that
$\overline{b^2}\big(R/J(R)\big)+\overline{e}\big(R/J(R)\big)=R/J(R).$
Write $\overline{b^2y+ez}=\overline{1}$ for some $y,z\in R$. Thus,
we can find a $v\in R$ such that $b^2yv+ezv=1$ in $R$. Since
$ezv\in 1+I$, we have an idempotent $f$ such that $f=ezvk$ and
$1-f=(1-ezv)l$ for some $k,l\in R$. It follows that $b^2yvl+f=1$,
and so $(1-f)b^2yvl+f=1$. We infer that $(1-f)b^2\in I$ is
regular.

By hypothesis, $\overline{(1-f)b^2}\in R/J(R)$ is strongly
regular. Thus,
$\overline{(1-f)b^2}\in R/J(R)$ is unit-regular~\cite{MM}. Since
$\overline{(1-f)b^2yl+f}=\overline{1}$, by virtue of
Lemma~\ref{32}, there exists some $t\in R$ such that
$\overline{(1-f)b^2+ft}\in U(R/J(R))$. That is,
$\overline{b^2+f(t-b^2)}\in U(R/J(R))$. As every unit lifts modulo
$J(R)$, we see that $b^2+f(t-b^2)\in U(R)$, and so
$b^2+ezk(t-b^2)\in U(R)$. Therefore $b^2+dszk(t-b^2)\in U(R)$, as
required.
\end{proof}

In view of Theorem~\ref{35}, we show that every exchange ideal of
an alelian ring is square stable. Furthermore, we can enhance
Khurana-Lam-Wang's theorem~\cite[Theorem 5.8]{KLW} as follows:

\begin{cor} Let $R$ be an exchange ring. Then the following are equivalent:
\end{cor}
\begin{enumerate}
\item [(1)]{\it $R$ has square stable range one.} \vspace{-.5mm}
\item [(2)]{\it For any regular $a\in R$, $\overline{a}\in R/J(R)$ is strongly regular.}
\end{enumerate}

We have at our disposal all the information necessary to prove the following result.

\begin{thm} \label{37} Let $I$ be an exchange ideal of a ring $R$. Then the following are equivalent:
\end{thm}
\begin{enumerate}
\item [(1)]{\it $I$ is square stable.}
\vspace{-.5mm}
\item [(2)]{\it For any regular $a\in I$, $a\in
a^2R$. and $aR$ is Dedekind-finite} \vspace{-.5mm}
\item [(3)]{\it Every regular element in $I$ is strongly regular.}
\end{enumerate}
\begin{proof}  $(1)\Rightarrow (2)$ Suppose $I$ is square stable. For any regular $a\in R$, $a=aca$ for a $c\in
R$. As $aR+(1-ac)R=R$, we have a $y\in R$ such that
$u:=a^2+(1-ac)y\in U(R)$. This shows that
$acu=ac\big(a^2+(1-ac)y\big)=aca^2=a^2$. Hence, $ac=a^2u^{-1}$,
and so $a=a^2u^{-1}a\in a^2R$.

Set $e=ac$. Then $aR=eR$, and so it will suffice to show that
$eRe$ is Dedekind-finite. Given $xy=e$ in $eRe$, then
$\overline{x}\in R/J(R)$ is regular. In light of Theorem~\ref{35},
$\overline{x}\in R/J(R)$ is strongly regular. Write
$\overline{x}=\overline{tx^2}~\mbox{for some}~t\in R.$ We may
assume that $t\in eRe$. As $\overline{xy}=\overline{e}$, we get
$\overline{tx^2y}=\overline{tx}$. Hence,
$\overline{tx}=\overline{e}$. This shows that $e-tx\in J(R)$, and
so $e-tx\in J(eRe)$. It follows that $tx=e-(e-tx)\in U(eRe)$.
Thus, $x\in eRe$ is left invertible. Clearly, $x\in eRe$ is right
invertible. This implies that $x\in U(eRe)$, and then $yx=e$.
Therefore $eRe$ is Dedekind-finite.

$(2)\Rightarrow (3)$ Let $a \in I$ be regular. By hypothesis,
$aR=a^2R$ and $aR$ Dedekind-finite. Construct a map $\varphi:
Ra\to Ra^2, ra\mapsto ra^2$. Then $\varphi$ is an $R$-epimorphism.
If $ra^2=0$, then $ra=0$, and so $\varphi$ is an $R$-monomorphism.
This implies that $Ra\cong Ra^2$. As $a\in R$ is regular, so is
$a^2\in R$. This, $Ra^2$ is a direct summand of $R$. Write
$Ra^2\oplus D=R$. Then $Ra=Ra\bigcap \big(Ra^2\oplus
D\big)=Ra^2\oplus Ra\bigcap D$. In view of ~\cite[Lemma 5.1]{KLW},
$Ra$ is Dedekind-finite. Hence, $D=0$. Therefore, $Ra=Ra^2$, and
then $a\in a^2R\bigcap Ra^2$, as required.

$(3)\Rightarrow (1)$ For any regular $a\in R$, $a\in R$ is
strongly regular. Hence, $\overline{a}\in R/J(R)$ is strongly
regular. In light of Theorem~\ref{35}, we complete the proof.
\end{proof}

As an immediate consequence, we drive that an exchange ring $R$ has square stable range one if and only if
every regular element in $R$ is strongly regular~\cite[Theorem 5.8]{KLW}.

\section{Regular ideals}

In this
section, we explore more explicit characterization of square
stable regular ideals. Such ideals are very enrich.

\begin{ex} If $n=\prod\limits_{i=1}^{m}p_i^{k_i}$ is the prime
power decomposition of the positive integer $n$, and $p_j$ is an
odd prime and $k_j=1$ for at least one $j\in \{ 1,\cdots ,m\}$,
then ${\Bbb Z}_n[i]=\{ a+bi~|~a,b\in {\Bbb Z}_n,i^2=-1\}$ has a nonzero square stable regular
ideal. This is obvious by ~\cite[Corollary
3.12]{O}.\end{ex}

Our starting point is the following technical lemma.

\begin{lem} \label{41}~\cite[Lemma 13.1.19.]{CH} Let $I$ be a regular
ideal of a ring $R$ and $x_1,x_2,\cdots,x_m\in I$. Then there
exists an idempotent $e\in I$ such that $x_i\in eRe$ for all
$i=1,2,\cdots,m$.\end{lem}

Recall that an ideal $I$ of a ring $R$ has stable range one
provided that $aR+bR=R$ with $a\in 1+I,b\in R\Longrightarrow
a+by\in U(R)$. It is known that a regular ideal $I$ has stable
range one if and only if $eRe$ is unit-regular for all idempotents
$e\in I$. Surprisingly, square stable ideals possess a similar characterization.

\begin{thm} \label{42} Let $I$ be a regular ideal of a ring $R$. Then the following are equivalent:
\end{thm}
\begin{enumerate}
\item [(1)]{\it $I$ is square stable.}
\vspace{-.5mm}
\item [(2)]{\it $eRe$ is strongly regular for all idempotents $e\in I$.}
\end{enumerate}
\begin{proof} $(1)\Rightarrow (2)$ Let $e\in I$ be an idempotent.
Then $eRe$ is regular. If $x^2=0$ in $eRe$, then $x\in J(R)$ by
Theorem~\ref{33}. As $x\in I$ is regular, we have a $y\in I$ such
that $x=xyx$, and then $x(1-yx)=0$. This implies that $x=0$.
Hence, $eRe$ is reduced. As $eRe$ is regular, $eRe$ is strongly
regular.

$(2)\Rightarrow (1)$ Suppose that $ax+b=1$ with $a\in I, x,b\in
R$. Then $b\in 1+I$, and so $a,1-b\in I$. By view of
Lemma~\ref{41}, there exists an idempotent $e\in I$ such that
$a,1-b\in eRe$. Write $a=ea'e$ and $1-b=eb'e$. Then $ea'ex+b=1$,
and so $(ea'e)(exe)+ebe=e$. Since $eRe$ is strongly regular, by
virtue of ~\cite[Theorem 5.2]{KLW}, $eRe$ has square stable range one. Hence, there
exists a $y\in R$ such that $(ea'e)^2+ebeye\in U(eRe)$. Thus, we
have a $u\in R$ such that
$\big((ea'e)^2+ebeye\big)(eue)=(eue)\big((ea'e)^2+ebeye\big)=e.$
This shows that
$$\begin{array}{lll}
\big((ea'e)^2+ebeye+1-e)(eue+1-e)&=&(eue+1-e)\big((ea'e)^2+ebeye+1-e\big)\\
&=&e+(1-e)\\
&=&1.
\end{array}$$
Clearly, $b(1-e)=1-e$ and $be=e-ea'exe=ebe$, and then
$$\begin{array}{lll}
\big(a^2+b(eye+1-e)\big)(eue+1-e)&=&(eue+1-e)\big(a^2+b(eye+1-e)\big)\\
&=&1.
\end{array}$$
Therefore $a^2+b(eye+1-e)\in U(R)$, and the result follows.\end{proof}

\begin{cor} \label{43} Let $I$ be regular ideal of a ring $R$. Then $I$ is square stable if and only if $I$ is reduced.\end{cor}
\begin{proof} Suppose that $I$ is square stable. If $x^2=0$ with $x\in I$, then there exists some idempotent
$e\in I$ such that $x\in eRe$, by Lemma~\ref{37}. In view of
Theorem~\ref{42}, $eRe$ is strongly regular, hence, it is reduced.
This implies that $x=0$, as desired.

Conversely, assume that $I$ is reduced. Then $eRe$ is reduced for
all idempotent $e\in I$. Hence, $eRe$ is strongly regular.
Therefore $I$ is square stable, in terms of Theorem~\ref{42}.
\end{proof}

We now characterize strongly regular rings
in terms of square stable ideals.

\begin{cor} Let $R$ be a regular ring. Then $R$ is strongly regular if and only if
\end{cor}
\begin{enumerate}
\item [(1)]{\it $I$ is square stable;}
\vspace{-.5mm}
\item [(2)]{\it $R/I$ is strongly regular;}
\vspace{-.5mm}
\item [(3)]{\it Every units of $R/I$ lifts to a unit of $R$.}
\end{enumerate}
\begin{proof} Suppose that $R$ is strongly regular. Then for any
idempotent $e\in I$, $eRe$ is strongly regular. In view of
Theorem~\ref{42}, $I$ is square stable. $(2)$ is obvious. Clearly,
$R$ is unit-regular. If $\overline{xy}=\overline{1}$. Then $x=xux$ for a $u\in U(R)$. Hence, $\overline{x}=\overline{u}^{-1}$.
$(3)$ holds.

Conversely, assume that $(1)-(3)$ hold. Given $ax+b=1$ with
$a,x,b\in R$, then $\overline{ax+b}=\overline{1}$ in $R/I$. By
$(2)$, $R/I$ has square stable range one, and then so does $R/I$. Thus, there exists a $y\in R$ such that
$\overline{a^2+by}\in U(R/I)$. By $(3)$, we have a $u\in U(R)$
such that $\overline{a^2+by}=\overline{u}$. Hence, $(a^2+by)u-1\in
I$. This shows that
$\big((a^2+by)u\big)\big(u^{-1}x\big)+b(1-yx)=1.$
Since $I$ is square stable, $eRe$ is strongly regular for all
idempotent $e\in I$. Hence, $I$ has stable range one. Thus, we can
find a $z\in R$ such that $(a^2+by)u+b(1-yx)z\in U(R);$ that is,
$a^2+b\big((1-yx)z+yu\big)\in U(R)$. Therefore $R$ has square range one.
In light of ~\cite[Theorem 5.4]{KLW}, $R$ is strongly regular.
\end{proof}

We now come to the main result of this section.

\begin{thm} \label{44} Let $I$ be a regular ideal of a ring $R$. Then the following are equivalent:
\end{thm}
\begin{enumerate}
\item [(1)]{\it $I$ is square stable;}
\vspace{-.5mm}
\item [(2)]{\it $aR+bR=R$ with $a\in 1+I,b\in R\Longrightarrow a^2+by\in U(R)$ for a $y\in R$.}
\end{enumerate}
\begin{proof} $(1)\Rightarrow (2)$ Given $aR+bR=R$ with $a\in 1+I,b\in R$, then we have $x,y\in R$ such that
$ax+by=1-a$. As $a-1\in I$, there exists an idempotent $e\in I$ such that $1-a=(1-a)e$. Hence,
$e-ae=axe+bye$. Clearly, $a(1-e)=1-e$, and so one easily checks that
$$\begin{array}{lll}
(eae)\big(e+exe)+ebye&=&eae(1+x)e+ebye\\
&=&ea\big(e+(1-e)\big)(1+x)e+ebye\\
&=&e\big(a(1+x)+by\big)e\\
&=&e.
\end{array}$$
Since $I$ is square stable, by virtue of Theorem~\ref{42}, $eRe$
is strongly regular. Thus, we have a $z\in eRe$ such that
$v:=(eae)^2+ebyez\in U(eRe)$. Let $w=(1-e)a^2e+(1-e)byz$.
Obviously, $(eae)^2=eaea=e\big(ae+a(1-e)\big)ae=ea^2e$, and that
$a^2(1-e)=a(1-e)=1-e$, and so $v=e(a^2+byz)e$. Hence,
$v+w=(a^2+byz)e$ and $1-e=(a^2+byz)(1-e)$. This shows that
$a^2+byz=v+w+1-e$. Clearly, $vw=v^{-1}w=w^2=w(1-e)=0$ and
$(1-e)w=w$. Hence, we check that
$(v+w+1-e)^{-1}=v^{-1}-wv^{-1}+1-e$, and then
$\big(a^2+byz\big)^{-1}=v^{-1}-wv^{-1}+1-e.$
Therefore, $a^2+byz\in U(R)$, as required.

$(2)\Rightarrow (1)$ Let $e\in I$ be an idempotent. Given $ax+b=e$ with $a,x,b\in eRe$, then $(a+1-e)(x+1-e)+b=1$ with $a+1-e\in 1+I$.
By hypothesis, we can find a $y\in R$ such that $u:=(a+1-e)^2+by\in U(R)$. This shows that
$u^{-1}\big((a+1-e)^2+by\big)=\big((a+1-e)^2+by\big)u^{-1}=1.$
As $(a+1-e)^2=a^2+1-e$, we see that $(1-e)u^{-1}=1-e$, and so $u^{-1}e=eu^{-1}e$. Therefore we have
$(eu^{-1}e)\big(a^2+b(eye)\big)=\big(a^2+b(eye)\big)(eu^{-1}e)=e.$ Accordingly, $a^2+b(eye)\in U(eRe)$. That is, $eRe$ is square stable.
In light of Theorem~\ref{42}, $I$ is square stable, hence the result.\end{proof}

\begin{cor} \label{45} Let $I$ be a regular ideal of a ring $R$. Then the following are equivalent:
\end{cor}
\begin{enumerate}
\item [(1)]{\it $I$ is square stable;}
\vspace{-.5mm}
\item [(2)]{\it Every element in $I$ is strongly regular.}
\vspace{-.5mm}
\item [(3)]{\it Every element in $1+I$ is strongly regular.}
\end{enumerate}
\begin{proof} $(1)\Leftrightarrow (2)$ As every regular ideal is
an exchange ideal, this is obvious by Theorem~\ref{37}.

$(1)\Rightarrow (3)$ Let $a\in 1+I$. Then $a-a^2\in I$. Since $I$ is regular, we see
that $a-a^2\in I$ is regular. Clearly, $I$ is an exchange ideal of
$R$. In view of Theorem~\ref{37}, $a-a^2\in R$ is strongly
regular. So $a-a^2=(a-a^2)^2x=y(a-a^2)^2$ for some $x,y\in R$, and
then $a=a^2\big(1+(1-a)^2x\big)=\big(y(1-a)^2+1\big)a^2$.
Therefore $a\in R$ is strongly regular.

$(3)\Rightarrow (1)$ Suppose that $ax+b$ with $a\in 1+I,x,b\in R$.
Then $a-a^2\in I$ is regular.
Thus, we can find some $x\in R$ such that
$a-a^2=(a-a^2)z(a-a^2)$. Hence, $a=a\big(a+(1-a)z(1-a)\big)a$, i.e., $a\in R$ is regular.
By hypothesis, $a\in R$ is strongly regular.
In view of ~\cite[Theorem 5.2]{KLW}, there exists a $y\in R$ such that $a^2+by\in U(R)$. This completes the proof, by Theorem~\ref{44}.
\end{proof}

Huanyin Chen

Department of Mathematics

Hangzhou Normal University

Hangzhou, 310036, China

Email: huanyinchen@aliyun.com\\

Marjan Sheibani Abdolyousefi

Department of Mathematics

Semnan University, Semann, Iran

Email: m.sheibani1@gmail.com\\

\end{document}